\documentclass[10pt]{amsart}
\title{The critical height is a moduli height}
\author{Patrick Ingram}
\email{pingram@yorku.ca}
\address{York University, 4700 Keele St., Toronto, Canada}
\date{\today}

\newcommand{\PP}{\mathbb{P}}
\newcommand{\ZZ}{\mathbb{Z}}
\newcommand{\CC}{\mathbb{C}}
\newcommand{\QQ}{\mathbb{Q}}
\newcommand{\RR}{\mathbb{R}}
\renewcommand{\AA}{\mathbb{A}}
\renewcommand{\epsilon}{\varepsilon}

\newcommand{\Rat}{\operatorname{Rat}}
\newcommand{\Res}{\operatorname{Res}}
\newcommand{\Hom}{\operatorname{Hom}}
\newcommand{\M}{\mathsf{M}}
\newcommand{\PGL}{\operatorname{PGL}}

\newcommand{\ord}{\operatorname{ord}}

\newcommand{\hcrit}{\hat{h}_{\mathrm{crit}}}

\newtheorem{theorem}{Theorem}
\newtheorem{lemma}[theorem]{Lemma}
\newtheorem{corollary}[theorem]{Corollary}

\theoremstyle{definition}
\newtheorem{remark}[theorem]{Remark}

\begin{document}
\begin{abstract}
Silverman defined the critical height of a rational function $f(z)$ of degree $d\geq 2$ in terms of the asymptotic rate of growth of the Weil height along the critical orbits of $f$, and conjectured that this quantity was commensurate to an ample Weil height on the moduli space of rational functions degree $d$. We prove this conjecture. 
\end{abstract}
\maketitle

If $f(z)\in\overline{\QQ}(z)$ is a rational function of degree $d\geq 2$, the canonical height $\hat{h}_f$ associated to $f$ is uniquely determined by the first two of its three fundamental properties~\cite[p.~99]{ads}:
\begin{enumerate}
\item[A.] $\hat{h}_f(f(P))=d\hat{h}(f)$,
\item[B.] $\hat{h}_f(P)=h_{\PP^1}(P)+O_f(1)$, with $h_{\PP^1}$ the usual Weil height, and
\item[C.] $\hat{h}_f(P)=0$ if and only if $P$ is preperiodic.
\end{enumerate}
These properties make the canonical height the natural choice of measure of arithmetic complexity on $\PP^1$, relative to $f$, and indeed it has been ubiquitous in the study of arithmetic dynamics.

Moving from dynamical space to the moduli space $\M_d$ of all rational functions of degree $d$, Silverman proposed a natural measure of the dynamical complexity of a rational function based on its critical orbits. In particular, the \emph{critical height} is defined by
\[\hat{h}_{\mathrm{crit}}(f)=\sum_{P\in \operatorname{Crit}(f)}\hat{h}_f(P),\]
where the sum is over critical points of $f$ with multiplicity,
a definition which is independent of choice of coordinates. The orbits of critical points carry a great deal of information about a dynamical system, and so this is a natural candidate for a ``canonical height'' on moduli space, and indeed it enjoys two properties similar to properties above:
\begin{enumerate}
\item[A$'$.] $\hat{h}_{\mathrm{crit}}(f^n)=n\hat{h}_{\mathrm{crit}}(f)$,
\item[C$'$.] $\hat{h}_{\mathrm{crit}}(f)=0$ if and only if $f$ is post-critically finite (PCF), that is, if and only if every critical orbit of $f$ is finite.
\end{enumerate}
One might hope that $\hat{h}_\mathrm{crit}$ turns out to be an ample Weil height on $\M_d$, in analogy with the second property of $\hat{h}_f$, but the algebraic geometry of $\M_d$ is more complicated than that of $\PP^1$, and besides the flexible Latt\`{e}s examples~\cite[\S6.5]{ads} show that this is just not true (these have critical height zero, but unbounded moduli height). Indeed, $\hat{h}_f$ is constructed directly from $h_{\PP^1}$, making property B above fairly unsurprising, while $\hat{h}_{\mathrm{crit}}$ is defined pointwise in terms of the dynamics of each map, with no mention of, or obvious relation to, the arithmetic geometry of $\M_d$. Nonetheless, Silverman conjectured~\cite[Conjecture~6.29, p.~101]{barbados} that $\hat{h}_{\mathrm{crit}}$ is commensurate to any ample Weil height on $\M_d$, away from the Latt\`{e}s examples, and our main result confirms this conjecture.

\begin{theorem}\label{th:main}
For any ample Weil height $h_{\M_d}$ on $\M_d$,
\[\hcrit\asymp h_{M_d}\]
except at the Latt\`{e}s examples.
In other words there exist positive constants $c_1, c_2, c_3$, and $c_4$ such that
\[c_1h_{\M_d}(f)-c_2\leq \hcrit(f)\leq c_3h_{\M_d}(f)+c_4\]
for all non-Latt\`{e}s $f\in \M_d$.
\end{theorem}

Note that, since $\hcrit(f)=0$ for any PCF map $f$, Theorem~\ref{th:main} contains as a special case the main result of~\cite{bijl}, namely that the non-Latt\`{e}s PCF locus is a set of bounded height in moduli space. We also note that the upper bound on $\hcrit(f)$ is relatively straightforward from estimates of the difference between the canonical height and the usual Weil height, and is contained  in~\cite[Theorem~6.31, p.~101]{barbados}. The lower bound on $\hcrit(f)$ is the new piece added here.

The proof goes roughly as follows (with estimates made precise in Sections~\ref{sec:greens} and~\ref{sec:global}). It is known, classically over $\CC$ and from~\cite{bijl} over non-archimedean fields, that any sufficiently attracting (but not super-attracting) fixed point linearly attracts a critical orbit. For each valuation $v$ and each fixed point $\gamma$ of multiplier $\lambda\neq 0$, there exists a critical point $\zeta$ for which this linear attaction gives an estimate  of the form
\[\log^+\left|\frac{1}{f^k(\zeta)-\gamma
}\right|_v\geq k\log^+|\lambda^{-1}|_v-C_1(f, k, v),\]
adjusting the constant to make the statement true even when $|\lambda|_v$ is not small, where $\log^+ x=\log\max\{1, x\}$. Summing over all critical points, all fixed points, and all places, we obtain an estimate of the form
\[k \left(h(\lambda_1) + \cdots +h(\lambda_{d+1})\right)\leq \sum_{f'(\zeta)=0} h(f^k(\zeta))+C_2(f, k),\]
where $\lambda_1, ..., \lambda_{d+1}$ are the multipliers at the fixed points (some of which may now vanish), and using an estimate $h(P)=\hat{h}_f(P)+O_f(1)$, we further obtain
\[k \left(h(\lambda_1) + \cdots +h(\lambda_{d+1})\right)\leq d^k\hcrit(f)+C_3(f, k).\]
This can be applied to some iterate to bound $\hcrit(f)$ below by the heights of multipliers of points of period dividng $n$, for any $n$, and this in turn may be bounded below in terms of $h_{\M_d}(f)$ for some $n$ depending just on $d$ (an idea used already in~\cite{bijl}, relying on McMullen's Theorem on stable families~\cite{mcmullen}, and made more precise below). Ultimately we end up with an estimate of the form
\[\epsilon k h_{\M_d}(f)\leq d^k\hcrit(f) + C_4(f, k),\]
for all $k\geq 1$, with $\epsilon>0$ depending just on $d$. Of course, with no information about the error term this estimate could easily be trivial, but it turns out that we may take \[C_4(f, k)=C_5 h_{\M_d}(f)+C_6k,\] with $C_5$ and $C_6$ depending just on $d$. Once $k$ is large enough, the contribution of $h_{\M_d}(f)$ in the lower bound exceeds that in the error term, giving a lower bound on $\hcrit(f)$ of the sort claimed in Theorem~\ref{th:main}.
 
The proof of Theorem~\ref{th:main} gives lower bounds on the critical height based on the height of the multiplier of any given periodic point, which is of interest in the study of certain fibrations of $\M_d$. In particular, the subvarieties $\operatorname{Per}_n(\lambda)\subseteq\M_d$ of rational functions with an $n$-cycle  of multiplier $\lambda$ are well-studied~\cite{demarco, milnor2, milnor}, and of notable interest is the distribution of PCF points on these subvarieties. It follows from the main result of \cite{bijl} that $\operatorname{Per}_n(\lambda)$ contains no PCF points at all once $h(\lambda)$ is large enough (for a fixed $n$) and so for $h(\lambda)$ large the function $\hcrit$ is non-vanishing on $\operatorname{Per}_n(\lambda)$. It is natural to ask whether or not a Bogomolov-type phenomenon occurs, in which $\hcrit$ has a positive infimum on $\operatorname{Per}_n(\lambda)$. Our next result gives a uniform asymptotic in this direction.
\begin{theorem}\label{th:fibration}
For fixed $n\geq 1$ and $d\geq 2$, there exist constants $\epsilon>0$ and $B$ such that 
\[\frac{\hcrit(f)}{h(\lambda)}\geq \epsilon>0\]
for all $f\in \operatorname{Per}_n(\lambda)$, for all $\lambda$ with $h(\lambda)> B$.
\end{theorem}
It follows from this that if $\hcrit$ takes arbitrarily small values on $\operatorname{Per}_n(\lambda)$, then $h(\lambda)$ is bounded, and the proof gives such a bound. The $m$th multiplier spectrum is the morphism taking $f\in \M_d$ to (the symmetric functions in) the multipliers of the fixed points of $f^m$, and by McMullen's Theorem on stable families~\cite[Corollary~2.3]{mcmullen} there exists an $m$ for which this morphism is finite away from the Latt\`{e}s locus. For this $m$, if $\hcrit$ takes arbitrarily small values on $\operatorname{Per}_1(\lambda)$, we have (from the proof of Theorem~\ref{th:fibration})
\[h(\lambda)\leq \frac{1}{n}\left(1.04d^{2mn}+d^{mn}\log 3-\log 3\right).\]

Although Theorem~\ref{th:main} resolves the conjecture made by Silverman, it is by no means easy to recover explicit values for the claimed constants from the proof, and so it is not obvious how one could effectively list, say, all $f(z)\in\QQ(z)$ of degree 4 and critical height at most 10 (although the theorem certainly says that this is a finite list, up to conjugacy). Note that the more elementary methods employed for polynomials in~\cite{pcfpoly} did offer this level of information, and here there are two special cases in which the constants can be made increasingly concrete.

The first case is that of rational functions with a super-attracting fixed point (generalizing the case of polynomials~\cite{pcfpoly}). 
Note that if we are to obtain any sort of explicit lower bound on $\hcrit(f)$ in terms of the coefficients of $f$, we will need to make some assumptions normalizing the choice of coordinate. Here $h_{\Hom_d}(f)$ is the height of the tuple of coefficients of $f$ as a point in $\PP^{2d+1}$. 
\begin{theorem}\label{th:geom}
For any $d\geq e\geq 2$ there exists an explicit constant $C_{d, e}$ such that if $\deg(f)=d$, $f(z)=z^e+O(z^{e+1})$ formally at $z=0$, and $f(\infty)=\infty$, then
\[\hcrit(f)\geq \left(\frac{1}{(d-1)d^2(4d^2-2(e+2)d+e+2)^{\log d/\log e}}\right) h_{\Hom_d}(f)-C_{d, e}.\]
\end{theorem}
Note that any rational function of degree $d$ with a fixed point of local degree $e\geq 2$ can be put in this form, since a super-attracting fixed point cannot be the only fixed point of $f$. On the other hand, the lower bound in the statement depends on the chosen form, since $h_{\Hom_d}$ is by no means constant on conjugacy classes. An explicit  value for $C_{d, e}$ appears in equation~\eqref{eq:exp} below. Examining the effect of change-of-coordinates on $f$, we can establish an inequality of the sort in Theorem~\ref{th:geom} for any rational function with a fixed point of local degree $e$, but the error term will depend (linearly) on the height of this fixed point, the height of the first nonzero coefficient of the Taylor series of $f$ at this point, and the height of one other fixed point (three data which suffice to fix a coordinate).

\begin{corollary}
For any $d\geq 2$, and any $B$, let $S\subseteq\M_d$ be the set of conjugacy classes of PCF rational functions of degree $d$ admitting a super-attracting fixed point, and admitting a model with coefficients of algebraic degree at most $B$. Then there is a finite and effectively computable list representatives of the classes in $S$.
\end{corollary}
The finiteness follows from \cite[Theorem~1.1]{bijl}; what is new in this corollary is the explicit nature of the bound. Although we focus on number fields, the sorts of estimates that prove Theorem~\ref{th:geom} also have consequences in function fields, as noted in Remark~\ref{rm:ff}.

We present one last case in which we can give completely concrete (although probably not sharp) bounds. Milnor~\cite{milnor} explicitly described the moduli space of quadratic morphisms, and (except for a single one-parameter family that can be handled easily on its own) the family defined for $(\lambda_0, \lambda_\infty)\in\AA^2$ by
\[f_{\lambda_0, \lambda_\infty}(z)=\frac{\lambda_0 z+z^2}{\lambda_\infty z+ 1}\]
offers a double-cover of $\M_2$. For computational purposes, it is preferable to work on this cover, and to replace $h_{\M_2}(f)$ by the more explicit, but comparable, $h_{\PP^2}(\lambda_0, \lambda_\infty)$.
\begin{theorem}\label{th:quad}
For all $\lambda_0, \lambda_\infty\in\overline{\QQ}$, we have
\[\hcrit\left(f_{\lambda_0, \lambda_\infty}\right)\geq \frac{1}{2048}h_{\PP^2}(\lambda_0, \lambda_\infty)-0.012.\]
\end{theorem}
Unfortunately, the constants above, while less intimidating than those in Theorem~\ref{th:geom}, are still generous enough that computing the smallest positive critical height on $\M_2(\QQ)$, say, might still be out of reach. 
\begin{corollary}\label{cor:quad}
Let $\operatorname{Per}_1(\lambda)\subseteq \M_2$ be the collection of quadratic morphisms with a fixed point of multiplier $\lambda\in\overline{\QQ}$, and suppose that
\[\inf_{f\in\operatorname{Per}_1(\lambda)}\hcrit(f)=0.\]
Then $h(\lambda)\leq \log 12$.
\end{corollary}

In Section~\ref{sec:greens}, we gather the key estimates on local Arakelov-Green's functions associated to endomorphisms of $\PP^1$. We sum these in Section~\ref{sec:global} to prove Theorem~\ref{th:main}, and in the process Theorem~\ref{th:fibration}. In Section~\ref{sec:sa} we take up the problem of obtaining explicit bounds for rational functions admitting super-attracting fixed points, proving Theorem~\ref{th:geom}, 
and in Section~\ref{sec:quad} we treat Theorem~\ref{th:quad}.

\subsection*{Acknowledgements} The author would like to thank Kenneth Jacobs, Joseph Silverman,  and David McKinnon for helpful comments on an earlier draft, as well as the anonymous referees for many important corrections.


\section{Estimates for Green's functions}\label{sec:greens}

In this section we let $K$ be an algebraically closed field of charactersitic 0 or $p>d$, complete with respect to the absolute value $|\cdot|$, which might be archimedean or not. In the archimedean case, it will suffice to consider the case $K=\CC$ (since every archimedean field is, up to scaling of the absolute value, a subfield of $\CC$). Recall that $\log^+ x = \log\max\{x, 1\}$, and note that if $n$ is an integer, then
\[\log^+|n|=\begin{cases}\log n&\text{ if $v$ is archimedean}\\ 0&\text{ otherwise.}\end{cases}\]
We use this to simplify notation in several places. Indeed, the triangle and ultametric inequalities appear often as
\[\log|x_1+\cdots +x_n|\leq \log\max\{|x_1|, ..., |x_n|\}+\log^+|n|.\]
We will also write $D(a, r)$ for the open disk of radius $r$ at $a$, that is,
\[D(a, r)=\{z\in K:|a-z|<r\}.\]

 Given a rational function $f(z)\in K(z)$ with $d=\deg(f)\geq 2$, we recall the construction of the dynamical Arakelov-Green's function $g_f:(\PP^1_K)^2\to \RR$ associated to $f$ (see \cite[\S10.2]{br} for more details). We first choose a pair of homogeneous forms $F_1, F_2\in K[x, y]$ with $f(x/y)=F_1(x, y)/F_2(x, y)$, and then define for $F=(F_1, F_2)$,
\[H_F(x, y)=\lim_{n\to\infty}d^{-n}\log\|F^n(x, y)\|,\]
where we always take $\|x_1, ..., x_m\|=\max\{|x_1|, ..., |x_m|\}$. We then set
\[g_f([x:y], [z:w])=-\log|yz-xw|+H_F(x, y)+H_F(z, w)-r(F),\]
for
\[r(F)=\frac{1}{d(d-1)}\log|\operatorname{Res}(F_0, F_1)|.\] This is easily shown to be independent of choice of representative homogeneous coordinates, and even of the choice of the forms $F_0$ and $F_1$ (but we will in fact impose a particular choice below).

 For the remainder of this section, we consider rational functions of the form
\begin{equation}\label{eq:normalform}
f(z)=\frac{\lambda z+\cdots + a_{d}z^{d}}{1+b_1z+\cdots +b_{d}z^{d}}=\frac{\lambda z\prod (1-\alpha_i z)}{\prod(1-\beta_j z)},
\end{equation}
and after the proof of Lemma~\ref{lem:greensineq} we will assume that $\lambda\neq 0$ and that $b_d=0$, ensuring that $f(\infty)=\infty$.
We set \[\|f\|=\max\{|\lambda|, |a_2|, \cdots, |a_d|, 1, |b_1|, \cdots, |b_d|\},\] noting that this implies $\log\|f\|\geq 0$. For convenience of notation, we will also set $a_1=\lambda$  and, in the early part of the next proof, refer to the constant terms in the numerator and denominator simply as $a_0$ and $b_0$ (it does not matter until later that $a_0=0$ and $b_0=1$). In terms of the construction of $g_f$ above, we  choose once and for all the obvious homogeneous lift \[F(x, y)=(\lambda xy^{d-1}+\cdots +a_dx^d, y^d+\cdots +b_dx^d),\]
allowing us to speak unambiguously about $r(f)=r(F)$.

The following lemma is related to \cite[Lemma~10.1, p.~294]{br}, \cite[Poposition~5.57, p.~288]{ads}, and similar results that have appeared elsewhere, but for our purposes we need error terms that are uniform in the coefficients of $f$.
\begin{lemma}\label{lem:greensineq}
For $f$ of the form~\eqref{eq:normalform} and for all $z\in \PP^1_K$,
\[g_f(z, 0)\geq \log^+|z^{-1}|-\frac{1}{d-1}\log^+|2d(2d-1)!|-\frac{2d-1}{d-1}\log\|f\|+(d-1)r(f).\]
\end{lemma}

\begin{proof}
This estimate is quite general (we do not assume $\lambda\neq 0$), and in large part standard. 

Let
\begin{gather}
F_1(x, y)=a_dx^d+\cdots +a_0y^d\nonumber\\
F_2(x, y)=b_dx^d+\cdots +b_0y^d,\nonumber
\end{gather}
with no other hypotheses on the coefficients. Using Cramer's Rule, we can solve
\begin{equation}\label{eq:bezout}G_1F_1+G_2F_2=\Res(F_1, F_2)x^{2d-1}\end{equation}
with $G_1$ and $G_2$ homogeneous forms of degree $d-1$ in $x$ and $y$, whose coefficients are homogeneous forms of degree $2d-1$ in the $a_i$ and $b_j$, each a sum of at most $(2d-1)!$ monomials of coefficient $\pm 1$. So we have
\[\log |G_i(x, y)|\leq (d-1)\log\|x, y\| + (2d-1)\log\|a_0, ..., b_{d}\|+\log^+|d(2d-1)!|,\]
and combining this with~\eqref{eq:bezout}, we obtain
\begin{multline*}
(2d-1)\log|x|+\log |\Res(F_1, F_2)|\leq \log \|F_1(x, y), F_2(x, y)\|\\+(d-1)\log\|x, y\| + (2d-1)\log\|a_0, ..., b_{d}\|+\log^+|2d(2d-1)!|.
\end{multline*}
Obtaining a similar upper bound on $(2d-1)\log|y|+\log|\Res(F_1, F_2)|$, we then derive
\begin{multline*}
\frac{1}{d}\log\|F(x, y)\|\geq \log\|x, y\|-\frac{2d-1}{d}\log\|a_0, ..., b_d\|\\-\frac{1}{d}\log^+|2d(2d-1)!|+\frac{1}{d}\log|\Res(F_1, F_2)|
\end{multline*}
for all $x$ and $y$ (not both 0). By induction, a bound of the form \[\frac{1}{d}\log\|F(x, y)\|\geq \log\|x, y\|-C\] for all $x$ and $y$ (not both 0) implies one of the form
\[\frac{1}{d^n}\log\|F^n(x, y)\|\geq \log\|x, y\|-\left(1+\frac{1}{d}+\cdots+\frac{1}{d^{n-1}}\right)C,\]

and so taking limits we have
\[H_F(x, y)\geq\log\|x, y\|-\frac{1}{d-1}\log^+|2d(2d-1)!|-\frac{2d-1}{d-1}\log\|f\|+dr(f).\]
(recalling that $\|f\|=\max\{|a_0|, ..., |b_d|\}$ and $r(f)=\frac{1}{d(d-1)}\log|\Res(F_1, F_2)|$ for our choice of lift $f(z)=F_1(z, 1)/F_2(z, 1)$). 

Now, returning to our hypotheses, our lift $F$ of $f$ satisfies $F(0, 1)=(0, 1)$ so $H_F(0, 1)=0$. By definition, for $z=[x:y]$ we have
\begin{eqnarray*}
g_f(z, 0)&=&g_f([x:y], [0:1])\\
&=&-\log|x|+H_F(x, y)-r(f)\\
&\geq &\log^+|z^{-1}|-\frac{1}{d-1}\log^+|2d(2d-1)!|-\frac{2d-1}{d-1}\log\|f\|+(d-1)r(f).
\end{eqnarray*}
\end{proof}

Our next lemma is essentially a classical result of Fatou, that any attracting cycle attracts a critical point, along with its $p$-adic analogue~\cite{bijl}, both given a slightly more explicit form.  For the statement, we define two constants $\epsilon_v>0$ and $C_v$, depending on $d$ and on the nature of the valuation $v$. We set
\[C_v=\begin{cases} 3^{d-1}&\text{if $v$ is archimedean,}\\
1&\text{otherwise,}
\end{cases}
\]
and
\begin{equation*}\epsilon_v=\begin{cases}\frac{1}{8}&\text{if $d=2$ and $v$ is archimedean,}\\
\frac{1}{C_v}&\text{if $d\geq 3$ and $v$ is archimedean,}\\
\min_{1\leq m\leq d}|m|^d&\text{otherwise.}
\end{cases}\end{equation*}
We note that if $v$ is non-archimedean, and not $p$-adic for any $p\leq d$, then $\epsilon_v=C_v=1$. This is the case, for example, for any place of the function field $\CC(X)$ of a variety $X/\CC$.

\begin{lemma}\label{lem:attraction}
Let $f$ be of the form~\eqref{eq:normalform} with $b_d=0$, and suppose that $0<|\lambda|<\epsilon_v$. Then there exists a branch point $\beta$ of $f$ with
\[0<|f^k(\beta)|\max\{|\alpha_1|, ..., |\beta_{d-1}|\}\leq \left(C_v|\lambda|\right)^k\]
for all $k\geq 1$, with $\alpha_i$ and $\beta_j$ as in~\eqref{eq:normalform}.
\end{lemma}

In the case $K=\CC$, we use a more-or-less standard argument from complex dynamics.
\begin{proof}[Proof of Lemma~\ref{lem:attraction} for $K=\CC$]
Note that both sides of the inequality are left fixed by the conjugacy $f(z)\mapsto \xi^{-1}f(\xi z)$, and so without loss of generality we may assume that $\max\{|\alpha_1|, ..., |\beta_{d-1}|\}=1$. Given this, if $0< |z|\leq \frac{1}{2}$, then
\[0\neq|f(z)|=|\lambda z|\frac{\prod_{i=1}^{d-1}|1-\alpha_iz|}{\prod_{j=1}^{d-1}|1-\beta_jz|}\leq |\lambda z|\frac{\left(\frac{3}{2}\right)^{d-1}}{\left(\frac{1}{2}\right)^{d-1}}=|\lambda z|3^{d-1} < |z|\]
by the triangle inequality, and by our choice of $\epsilon_v$. Note that $f(z)\neq 0$ because every nonzero root $\alpha_i^{-1}$ of $f$ has $|\alpha_i^{-1}|\geq 1$.
By induction,
\[0\neq |f^k(z)|\leq \left(C_v|\lambda|\right)^k\]
for all $k\geq 1$ and all nonzero $z\in D(0, \frac{1}{2})=\{z\in \CC:|z|< \frac{1}{2}\}$. It remains to show that $D(0, \frac{1}{2})$ contains a nonzero branch point of $f$.

Suppose that $f$ has no branch point in the disk $D(0, \frac{1}{2})$, and let $W$ be the connected component of $f^{-1}(D(0, \frac{1}{2}))$ containing $0$. Since $f(W)= D(0, \frac{1}{2})$ we know that $W$ contains no poles of $f$, and since $f:W\to D(0, \frac{1}{2})$ is unbranched we know that $W$ contains no zeros other than $z=0$, and that $W$ is simply connected. But $f$ has at least one pole or nonzero root on the unit circle, and so $W$ does not contain the closed unit disk. By Koebe's $\frac{1}{4}$ Theorem, $W$ has conformal radius no greater than $4$ (relative to the origin). On the other hand, the map $f:W\to D(0, \frac{1}{2})$ witnesses that $W$ has conformal radius exactly $(2|\lambda|)^{-1}$, whereupon $|\lambda|\geq \frac{1}{8}$. This contradicts our hypothesis that $|\lambda|<\epsilon_v$.
\end{proof}

\begin{proof}[Proof of Lemma~\ref{lem:attraction} for $K$ non-archimedean]
Here we use the main result of \cite{bijl}. Again scaling coordinates, we may take
 $\max\{|\alpha_1|, ..., |\beta_d|\}=1$. By the ultrametric inequality, for any nonzero $z\in D(0, 1)$ we have
\[0\neq |f^k(z)|=|\lambda f^{k-1}(z)|\frac{\prod_{i=1}^{d-1}|1-\alpha_if^{k-1}(z)|}{\prod_{i=1}^{d-1}|1-\beta_jf^{k-1}(z)|}=|\lambda f^{k-1}(z)|=|\lambda|^k|z|\leq |\lambda|^k\]
for all $k\geq 1$. It is now enough to show that $D(0, 1)$ contains a nonzero branch point of $f$. This follows from~\cite[Theorem~4.1]{bijl}. Specifically, as in the preamble to the proof of that theorem, we have chosen coordinates so that $z=0$ is the fixed point with multiplier $\lambda$ and that $z=\infty$ is also a fixed point of $f$. We have also scaled the coordinate so that the smallest pole or nonzero root of $f$ has absolute value $1$, noting that $\alpha_1, ..., \beta_{d-1}$ are the reciprocals of these roots and poles. By the proof of~\cite[Theorem~4.1]{bijl}, there is now a branch point of $f$ in $D(0, 1)$ as long as
\[0<|\lambda|<|\deg_{\zeta, \vec{w}}f|^d\]
for all $\zeta$ in the Berkovich analytic space $\PP^1_{\mathrm{Berk}}$ and all tangent directions $\vec{w}$. But the directional multiplicity $\deg_{\zeta, \vec{w}} f$ is an integer between $1$ and $d$, and so this condition is ensured by $0<|\lambda|<\epsilon_v$.
\end{proof}

Before proceeding with the key lemma of this section, we note a standard relation between the sizes of the roots of a polynomial and those of its coefficients.
\begin{lemma}\label{lem:newtonP}
Let $e_1, ..., e_k\in K$, and suppose that
\[(z-e_1)(z-e_2)\cdots (z-e_k)=z^k+c_{k-1}z^{k-1}+\cdots +c_0.\]
Then
\begin{equation}\label{eq:rootsupper}\log\|e_1, ..., e_k\|\leq \log^+\|c_{k-1}, ..., c_0\| +\log^+|2|\end{equation}
and
\begin{equation}\label{eq:rootslower}\log\|c_{k-1}, ..., c_0\|\leq k\log^+\|e_1, ..., e_k\|+k\log^+|2|.\end{equation}
\end{lemma}

\begin{proof}
To prove~\eqref{eq:rootslower}, note that $c_i$ is (up to sign) the elementary symmetric polynomial of degree $k-i$ in the variables $e_1, ..., e_k$. Since this symmetric polynomial is the sum of $\binom{k}{k-i}=\binom{k}{i}$ monomials of total degree $k-i$, we have
\[\log|c_i|\leq (k-i)\log\|e_1, ..., e_k\|+\log^+\left|\binom{k}{i}\right|.\]
Replacing $\log\|e_1, ..., e_k\|$ with $\log^+\|e_1, ..., e_k\|$ only weakens the bound.
From this we have
\[\log |c_i|\leq k\log^+\|e_1, ..., e_k\|+\log^+\left|\binom{k}{i}\right|.\]
The claimed inequality follows from the fact that $\binom{k}{i}\leq 2^k$.

For non-archimedean absolute values, \eqref{eq:rootsupper} follows from the theory of Newton Polygons. More directly, if~\eqref{eq:rootsupper} fails in the non-archimedean case, then some $e_i$ satisfies $|e_i|>|c_j|$ for all $j$, and $|e_i|>1$. It follows then that $|c_je_i^j|<|e_i^k|$ for each $0\leq j< k$, and so
\[0=|e_i^k+c_{k-1}e_i^{k-1}+\cdots +c_0|=|e_i^k|>1,\]
which is impossible. For archimedean absolute values, we use an argument of Fujiwara~\cite{fujiwara}. If we assume that $w_j|e_i^k|> |c_je_i^j|$ for all $j$ (and some choice of weights $w_j$), we derive a contradiction to
\[0=|e_i^k+c_{k-1}e_i^{k-1}+\cdots +c_0|\geq |e_i^k|-\sum_{j=0}^{k-1}|c_je_i^j|\]
as soon as $\sum_{j=0}^{k-1}w_j<1$. In particular, we obtain a contradiction if we take $w_j=2^{j-k}$, which translates to the assumption that for some $i$ and all $j$ we have $|e_i|>2|c_j|^{1/(k-j)}$. To avoid the contradiction, then, we must have for each $i$ either $e_i=0$ or
\[\log|e_i|\leq\log^+\|c_{k-1}, c_{k-2}^{1/2}, ..., c_0^{1/k}\| +\log^+|2|.\]
Since $r^{1/j}<r$ when $r, j>1$, we deduce~\eqref{eq:rootsupper} from this.
\end{proof}

Now, for notational convenience, we extend $g_f(\cdot, 0)$ linearly to divisors. In other words, if $D=\sum m_P[P]$, then $g_f(D, 0)=\sum m_Pg_f(P, 0)$. We write $B_f$ for the branch locus of $f$, and $f_*^kB_f$ for its $k$th iterated forward image (which is the $(k+1)$th iterated forward image of the critical divisor). We write $B_f'$
 for the part of the branch locus not consisting of iterated preimages of $0$. In other words, if $e_P(f)$ is the index of ramification of $f$ at $P$, and $\mathcal{O}_f^{-}(Q)$ is the backward orbit of $Q$ under $f$, then
\[B_f'=\sum_{P\not\in \mathcal{O}_f^{-}(0)}\left(e_P(f)-1\right)[f(P)].\]

\begin{lemma}\label{lem:key}
For any $f$ of the form~\eqref{eq:normalform}, with $b_d=0$ and $\lambda\neq 0$, and for any $k\geq 1$, we have
\begin{multline}\label{eq:key}g_f(f_*^kB_f', 0)\geq (k-1)\log^+|\lambda^{-1}|+k\log\epsilon_v+\log\|\alpha_1, ..., \beta_{d-1}\|
\\-\log\|f\|-\log^+|2|\\
-\deg(B_f')\left(\frac{2d-1}{d-1}\log\|f\|+\frac{1}{d-1}\log^+|2d(2d-1)!|-(d-1)r(f)\right).\end{multline}
\end{lemma}

\begin{proof}
By Lemma~\ref{lem:greensineq},  we have
\begin{multline}\label{eq:baseest}
g_f(f^k(\beta), 0)\geq \log^+\left|\frac{1}{f^k(\beta)}\right|\\-\frac{1}{d-1}\log^+|2d(2d-1)!|-\frac{2d-1}{d-1}\log\|f\|+(d-1)r(f)
\end{multline}
for every branch point $\beta$ with $f^k(\beta)\neq 0$.

We first assume that $0<|\lambda|<\epsilon_v\leq 1$, and so by Lemma~\ref{lem:attraction} there exists a $\beta\in \operatorname{Supp}(B_f')$ with
\begin{equation}\label{eq:critorb}|f^k(\beta)|\leq (C_v|\lambda|)^k/\max\{|\alpha_1|, ..., |\beta_{d-1}|\}.\end{equation}
Note that $-\log C_v\geq \log\epsilon_v$ for all $v$. For this branch point, we apply~\eqref{eq:baseest} along with the estimate $\log^+|z|\geq \log|z|$ and the estimate in~\eqref{eq:critorb}.
For every other branch point, we apply~\eqref{eq:baseest} with the trivial estimate $\log^+|z|\geq 0$ to obtain
\begin{multline}\label{eq:maincase}
g_f(f_*^k B_f', 0)\geq k\log^+|\lambda^{-1}|+k\log \epsilon_v+\log\|\alpha_1, ..., \beta_{d-1}\|\\-\deg(B_f')\left(\frac{1}{d-1}\log^+|2d(2d-1)!|+\frac{2d-1}{d-1}\log\|f\|-(d-1)r(f)\right).
\end{multline}
To obtain~\eqref{eq:key} in this case, it is enough to note that
\[\log^+|\lambda^{-1}|+\log\|f\|+\log^+|2|\geq 0,\]
and that subtracting the left-hand-side of this from the lower bound in~\eqref{eq:maincase} gives~\eqref{eq:key}.

It remains to show that~\eqref{eq:key} still holds when $|\lambda|\geq \epsilon_v$. If we apply~\eqref{eq:baseest} for each branch point, with the trivial estimate $\log^+|z|\geq 0$, we have
\begin{multline*}
g_f(f_*^k B_f', 0)\geq -\deg(B_f')\left(\frac{1}{d-1}\log^+|2d(2d-1)!|+\frac{2d-1}{d-1}\log\|f\|-(d-1)r(f)\right).
\end{multline*}
Comparing this with the desired inequality~\eqref{eq:key}, we note that it is enough to show that
\[0\geq (k-1)\log^+|\lambda^{-1}|+k\log \epsilon_v+\log\|\alpha_1, ..., \beta_{d-1}\|- \log\|f\|-\log^+|2|.\]
Our hypothesis in this case ensures that $k(\log^+|\lambda^{-1}|+\log \epsilon_v)\leq 0$, so we are left with showing
\[\log\|\alpha_1, ..., \beta_{d-1}\|\leq \log^+|\lambda^{-1}|+ \log\|f\|+\log^+|2|.\]

Since
\begin{gather*}
(z-\beta_1)\cdots(z- \beta_{d-1})=z^{d-1}+b_1z^{d-2}+\cdots +b_{d-1}\\
(z-\alpha_1)\cdots(z- \alpha_{d-1})=z^{d-1}+\frac{a_2}{\lambda}z^{d-2}+\cdots+\frac{a_d}{\lambda}
\end{gather*}
we may apply~\eqref{eq:rootsupper} of Lemma~\ref{lem:newtonP} twice to obtain
\begin{eqnarray*}
\log\|\beta_1, ..., \beta_{d-1}\|&\leq & \log^+\|b_1, ..., b_{d-1}\|+\log^+|2|
\\
\log\|\alpha_1, ..., \alpha_{d-1}\|&\leq & \log^+\|a_2/\lambda, ..., a_{d}/\lambda\|+\log^+|2|\\
&\leq &\log^+\|a_2, ..., a_d\|+\log^+|\lambda^{-1}|+\log^+|2|,
\end{eqnarray*}
and so
\[\log\|\alpha_1, ..., \beta_{d-1}\|\leq \log\|f\| + \log^+|\lambda^{-1}|+\log^+|2|,\]
as claimed, noting that $\log^+\|f\|=\log\|f\|$.
\end{proof}

Note that it is \emph{a priori} possible that $B_f'$ is the zero divisor, but this does not present a problem for the previous lemma. In this case our definitions give $g_f(f^kB_f', 0)=0$, and the lower bound simplifies to a bound on $|\lambda^{-1}|$ for PCF maps, as proved in~\cite{bijl}.

\section{Heights on $M_d$}\label{sec:global}

With local estimates in place, we turn our attention to global heights. In this section, we fix a number field $K$, although all estimates will remain unchanged after any finite extension, and so we are in some sense always working over $\overline{\QQ}$. Let $M_K$ be the standard set of valuations on $K$, with the absolute value $|\cdot|_v$ normalized to restrict to $\QQ$ as either the usual or one of the $p$-adic absolute values. With absolute values thus normalized, the usual Weil height is
\[h(\alpha)=\sum_{v\in M_K}\frac{[K_v:\QQ_v]}{[K:\QQ]}\log^+|\alpha|_v.\]
 Heights on projective varieties are defined relative to line bundles, with an \emph{ample height} being one defined relative to an ample bundle (see, e.g., \cite[Defintion~10.4, p.~142]{vojta}).
Quantities from Section~\ref{sec:greens} which depended on the absolute value now acquire a subscript $v$.

 Recall (from~\cite[\S4.3, 4.4]{ads} or~\cite[Ch.~1, 2]{barbados}) the space $\Rat_d=\PP^{2d+1}$ of rational functions of degree at most $d$, with \[\mathbf{c}=[c_0:\cdots :c_{2d+1}]\] corresponding to the rational function
\[f_{\mathbf{c}}(z)=\frac{c_0+c_1z+\cdots +c_d z^d}{c_{d+1}+c_{d+2}z+\cdots +c_{2d+1}z^d}.\]
The resultant of the numerator and denominator of $f_\textbf{c}$ cuts out a hypersurface in $\PP^{2d+1}$, and the complement of this is $\Hom_d\subseteq \Rat_d$ consisting of those rational functions of degree exactly $d$. There is a natural action of $\PGL_2$ on $\Hom_d$, by change of coordinates, and the quotient $\M_d$ is an affine variety parametrizing coordinate-free dynamical systems of degree $d$. By some abuse of notation, we will use the same symbol $f$ to identify a rational function, the point representing it in $\Rat_d$, and the point representing its conjugacy class in $\M_d$. When we write $h_{\M_d}$, we mean the height relative to some ample line bundle on $\M_d$. Any two such functions will be commensurate (for instance, by Lemma~\ref{lem:ampledominates}).

Much more concretely, $\Hom_d$ as a subvariety of $\PP^{2d+1}$ carries a natural height $h_{\Hom_d}$, which is just the usual Weil height
\[h_{\Hom_d}(f)=\sum_{v\in M_K}\frac{[K_v:\QQ_v]}{[K:\QQ]}\log\|f\|_v\]
where $\|f_{\mathbf{c}}\|_v=\max\{|c_0|_v, ..., |c_{2d+1}|_v\}$ as in Section~\ref{sec:greens}. Note that, even without the normalization in~\eqref{eq:normalform}, this is well-defined by the product formula
\[\sum_{v\in M_K}\frac{[K_v:\QQ_v]}{[K:\QQ]}\log|\alpha|_v=0\]
for $\alpha\neq 0$.

Finally, we recall Silverman's critical height. To each $f\in\Hom_d$ is associated a non-negative \emph{canonical height} $\hat{h}_f:\PP^1\to \RR$ defined by
\[\hat{h}_f(P)=\lim_{n\to \infty}\frac{h\circ f^n(P)}{d^n}.\]
Note that we can also decompose the canonical height locally as
\[\hat{h}_f(P)+\hat{h}_f(Q)=\sum_{v\in M_K}\frac{[K_v:\QQ_v]}{[K:\QQ]}g_{f, v}(P, Q)\]
for $P\neq Q$ \cite[p.~310]{br}.
We extend $\hat{h}_f$ linearly to divisors,
and set, for $C_f$ the critical divisor of $f$,
\[\hcrit(f)=\hat{h}_f(C_f)=\sum_{P\in\PP^1}(e_P(f)-1)\hat{h}_f(P),\]
where $e_P(f)$ is the ramification index of $f$ at $P$. It is easy to show that $\hcrit$ is invariant under change of coordinates, giving a well-defined non-negative function $\hcrit:\M_d\to \RR$. Also, since $\hat{h}_f\circ f=d\hat{h}_f$, we note that $\hat{h}_f(f_*D)=d\hat{h}_f(D)$ for any divisor $D$, and so in particular the branch locus $B_f$, and the divisor $B_f'$ defined in Section~\ref{sec:greens} by excising preimages of a certain fixed point, satisfy \begin{equation}\label{eq:branchheight}\hat{h}_f(f_*^{k}B_f)=\hat{h}_f(f_*^{k}B_f')=d^{k+1}\hat{h}_{\mathrm{crit}}(f)\end{equation} for any $k\geq 0$, since $\hat{h}_f$ vanishes at the points removed from $B_f$ to construct $B_f'$.

Because the results in Section~\ref{sec:greens} depend on the way in which $f$ is written, we must first show that we can change coordinates without changing the estimates too much.

\begin{lemma}\label{lem:goodconj}
Suppose that $f\in \Hom_d$ has a fixed point with multiplier $\lambda\neq 1$. Then there exists a $g\in \Hom_d$, conjugate to $f$, such that $g(0)=0$ with multiplier $\lambda$, $g(\infty)=\infty$, and
\begin{equation}\label{eq:goodconj}h_{\Hom_d}(g)\leq (d+2)h_{\Hom_d}(f) + (d+1)^2\log 2 + \log(d+1)(d+2).\end{equation}
\end{lemma}

\begin{proof}
First, we claim that if $\psi$ is any M\"{o}bius transformation, and $f^\psi = \psi^{-1}\circ f\circ \psi$, then
\begin{equation}\label{eq:mobiusheights}h_{\Hom_d}(f^\psi)\leq h_{\Hom_d}(f)+(d+1)h(\psi)+\log(d+1)(d+2),\end{equation}
where $h(\psi)$ is the height of the coefficients of $\psi$ as a  point in $\PP^3$. This can either be checked directly, by explicitly bounding the coefficients of $f^\psi$ in terms of those of $f$ and $\psi$, or a variant with a slightly worse error term can be derived from the general bound for the height of a composition of rational functions found in \cite[Proposition~5c]{hs}. (Even using the latter general bound, the error term in~\eqref{eq:goodconj} remains $O(d^2)$.)

Now, if $f$ has a fixed point at $[\alpha:\gamma]$ with multiplier $\lambda\neq 1$, then in particular $\alpha/\gamma$ is a simple root of $f(z)-z$, and so the latter equation must have at least one other root, say the fixed point $[\beta:\delta]$. If $\psi(z)=(\alpha z+\beta)/(\gamma z+\delta)$, then $g=f^\psi$ is a conjugate of $f$ with $g(\infty)=\infty$,  $g(0)=0$, and $g'(0)=\lambda$. It remains to estimate $h(\psi)$, and thereby $h_{\Hom_d}(g)$.

Note that the numerator of $f(z)-z$ has a coefficients of height at most $h_{\Hom_d}(f)+\log 2$ (as a tuple in $\PP^{d+1}$), and so by \cite[Theorem~5.9, p.~230]{aec} the fixed points $P_1, ..., P_{d+1}$ of $f$ satistfy $\sum h(P_i)\leq h_{\Hom_d}(f)+(d+1)\log 2$.
It follow that
\begin{eqnarray*}
h(\psi)&=&h([\alpha:\beta:\gamma:\delta])\\
&\leq &h([\alpha:\beta])+h([\gamma:\delta])\\
&\leq &h_{\Hom_d}(f)+(d+1)\log 2.
\end{eqnarray*}
 Combining this estimate with~\eqref{eq:mobiusheights} gives the claimed bound on $h_{\Hom_d}(g)$.
\end{proof}

The next lemma is a global version of Lemma~\ref{lem:key}, and in some sense is the crux of the main result.
\begin{lemma}\label{lem:firstbound}
Let $f\in\Hom_d$ have a fixed point with multiplier $\lambda$. Then for any $k\geq 1$
\begin{equation}\label{eq:mainglobal}d^{k+1}\hcrit(f)\geq (k-1)h(\lambda)-(4d-1)(d+2)h_{\Hom_d}(f)-c_0k\end{equation}
for some explicit positive constant $c_0$ depending just on $d$.
\end{lemma}

\begin{proof}
First, note that the inequality certainly holds when $\lambda=0$ or $\lambda=1$, since $h(0)=h(1)=0$ while $\hcrit$ and $h_{\Hom_d}$ are both non-negative, so we will assume that $\lambda\neq 0, 1$. We will first treat the case in which $f$ has the form ~\eqref{eq:normalform}, using Lemma~\ref{lem:key} to derive an even stronger estimate in this case.
Specifically, we sum the estimate~\eqref{eq:key} from Lemma~\ref{lem:key} over all places. Note that since $f(0)=0$ we have $\hat{h}_f(0)=0$, and so by~\eqref{eq:branchheight} we derive the following identities
\begin{gather}
\sum_{v\in M_K}\frac{[K_v:\QQ_v]}{[K:\QQ]} g_{f, v}(f^k_*B_f', 0)=\hat{h}_f(f_*^kB_f')+\deg(B_f')\hat{h}_f(0)=d^{k+1}\hcrit(f)\label{eq:hcritsum}\\
\sum_{v\in M_K}\frac{[K_v:\QQ_v]}{[K:\QQ]} \log^+|\lambda^{-1}|_v=h(\lambda^{-1})=h(\lambda)\label{eq:lambdasum}\\
\sum_{v\in M_K}\frac{[K_v:\QQ_v]}{[K:\QQ]} \log\|\alpha_1, 
..., \beta_{d-1}\|_v =h\left([\alpha_1:\cdots :\beta_{d-1}]\right)\geq 0\nonumber\\
\sum_{v\in M_K}\frac{[K_v:\QQ_v]}{[K:\QQ]} \log\|f\|_v = h_{\Hom_d}(f)\label{eq:homdsum}\\
\sum_{v\in M_K}\frac{[K_v:\QQ_v]}{[K:\QQ]}\log\epsilon_v = -d\log\operatorname{lcm}(1, ... ,d)-\log\max\{8,  3^{d-1}\}\label{eq:epsilonsum}\\
\sum_{v\in M_K}\frac{[K_v:\QQ_v]}{[K:\QQ]} \log^+|N|_v = \log N\label{eq:Nsum}\\
\intertext{for any integer $N$, and}
\sum_{v\in M_K}\frac{[K_v:\QQ_v]}{[K:\QQ]}r_v(f)=\frac{1}{d(d-1)}\sum_{v\in M_K}\frac{[K_v:\QQ_v]}{[K:\QQ]}\log|\Res(F_1, F_2)|_v=0,\label{eq:ressum}
\end{gather}
with this last equality following from  the product formula and $\Res(F_1, F_2)\neq 0$. Combining these identities with~\eqref{eq:key}, and using $\deg(B_f')\leq 2d-2$, we have
\begin{multline}\label{eq:fixedzero}d^{k+1}\hcrit(f)\geq (k-1)h(\lambda)-(4d-1)h_{\Hom_d}(f)\\-2\log 2d(2d-1)!-\log2 -kd\log\operatorname{lcm}(1, ..., d)-k\log\max\{8,  3^{d-1}\}.\end{multline}
Again, this holds only in the case where $f$ is of the form~\eqref{eq:normalform}, but the estimate depends only on the coefficients of $f$ as a point in $\Hom_d\subseteq\PP^{2d+1}$, so the choice of representative of homogeneous coordinates in~\eqref{eq:normalform} no longer matters. In obtaining~\eqref{eq:fixedzero}, then, we are using only that $f$ has a fixed point of multiplier $\lambda\neq 0$ at $z=0$, and another fixed point at $z=\infty$.

But if $f$ has a fixed point anywhere of multiplier $\lambda\neq 0,1$, let $g$ be the conjugate produced in Lemma~\ref{lem:goodconj}. Then applying~\eqref{eq:fixedzero} to $g$, and then~\eqref{eq:goodconj} to bound $-h_{\Hom_d}(g)$ below in terms of $-h_{\Hom_d}(f)$, we have the estimate~\eqref{eq:mainglobal}.
\end{proof}
Lemma~\ref{lem:firstbound} has the virtue of being completely explicit, but it is a lower-bound on a conjugacy-class invariant with an error that can get arbitrarily bad within a conjugacy class. The following lemma, due to Silverman, relates $h_{\Hom_d}$ to the height on $\M_d$, allowing coordinate-free estimates.
\begin{lemma}[Silverman~{\cite[p.~103]{barbados}}]\label{lem:silv}
For $f\in \M_d$, 
\[h_{\M_d}(f)\asymp \min_{g\sim f} h_{\Hom_d}(g),\]
where the minimum is over $g\in\Hom_d$ in the same conjugacy class as $f$.
\end{lemma}

Combining Lemma~\ref{lem:silv} with Lemma~\ref{lem:firstbound} gives an inequality in which no term is coordinate dependent, as we will see below.

We now construct, more explicitly, the height on $\M_d$ used in~\cite{bijl}.
Let $\Hom_d^{\operatorname{Fix}}$ be the space of rational functions with all fixed points $\gamma_1, ..., \gamma_{d+1}$ marked, and consider the map $\Hom_d^{\operatorname{Fix}}\to \AA^{d+1}\subseteq (\PP^1)^{d+1}$ by \[(f, \gamma_1, ..., \gamma_{d+1})\mapsto (\lambda_0, ..., \lambda_{d+1}),\]
where $\lambda_i$ is the multiplier at the fixed point $\gamma_i$. 
Since change of coordinates on $\Hom_d^{\operatorname{Fix}}$ acts by permuting the multipliers, this map induces a morphism $\sigma:\M_d\to S^{d+1}\PP^1$, where $S^{d+1}\PP^1$ denotes the $(d+1)$st symmetric power of $\PP^1$. Note that the image of the pull-back map $(\PP^1)^{d+1}\to S^{d+1}\PP^1$ on divisors is exactly the subgroup of the form $\sum \pi_i^* D$, where $\pi_i:(\PP^1)^{d+1}\to\PP^1$ is the $i$th coordinate projection.
In particular,
\[h_{S^{d+1}\PP^1}(\lambda_0, ..., \lambda_{d+1})=h_{\PP^1}(\lambda_0)+\cdots +h_{\PP^1}(\lambda_{d+1})\]
is an ample height on $S^{d+1}\PP^1$, and any ample Weil height is a scalar multiple of this (up to $O(1)$).

Now, for each $n$ we define a morphism $\sigma_n:\M_d\to S^{d^n+1}\PP^1$ by composing $\sigma$ with the iteration map $\M_d\to\M_{d^n}$. In other words, $\sigma_n(f)=\sigma(f^n)$. In~\cite{bijl}, we used (more-or-less) $h_{S^{d^n+1}\PP^1}\circ \sigma_n$ as a height on $\M_d$, which we justify more explicitly in the following lemma. Note that Latt\`{e}s maps are defined in~\cite[\S 6.5]{ads}, and they must be excluded in the following lemma, since they represent curves in $\M_d$ on which $\sigma_n$ is constant for any $n$ \cite[Proposition~6.52, p.~358 and Exercise~6.18, p.~382]{ads}.
\begin{lemma}\label{lem:sigmapullback}
For some $n$ depending only on $d$, we have
\[h_{\M_d}(f)\ll h_{S^{d^n+1}\PP^1}\circ\sigma_n(f)\]
for all non-Latt\`{e}s $f\in M_d$.
\end{lemma}

We deduce this lemma from a more general result. Note that this is closely related to~\cite[Proposition~10.13, p.~145]{vojta} and~\cite[Theorem~1]{jhsequi}, and can be deduced from either. Here we present a proof using \cite[Theorem~1]{jhsequi}.

\begin{lemma}\label{lem:ampledominates}
Let $X$ and $Y$ be irreducible projective varieties equipped with ample line bundles $L$ and $M$ (respectively), and let $U\subseteq X$ be a Zariski open subset with a morphism $F:U\to Y$ with finite fibres. Then for all $u\in U$,
\[h_{X, L}(u)\ll h_{Y, M}(F(u)).\]
\end{lemma}

\begin{proof}
We proceed by induction on $\dim(X)$, noting the the claim is trivial if $\dim(X)=0$. So now suppose that the claim in the lemma is true in all cases where the domain has dimension less than $\dim(X)$. Note that for any closed $W\subseteq Y$, the restriction $M|_W$ of $M$ to $W$ is ample, and $h_{W, M|_W}$ is the restriction to $W$ of $h_{Y, M}$. Without loss of generality, we may replace $Y$ by the Zariski closure  $\overline{F(U)}$, and $M$ by its restriction to this subvariety, and thereby assume that $F$ is dominant. If $U$ is non-empty, then $F$ has at least one finite fibre, and so $\dim(U)=\dim(Y)$. By~\cite[Theorem~1]{jhsequi}, there exist constants $C_{1, 0}$ and $C_{2, 0}$ and a Zariski-closed subset $Z\subseteq X$ such that
\[h_{X, L}(u)\leq C_{1, 0}h_{Y, M}(F(u))+C_{2, 0}\]
for all $u\in U\setminus Z$. Of course, if $U$ is empty then such an inequality holds vacuously, with the exceptional set $Z$ also empty.

Now, $Z$ has finitely many irreducible components $Z_1, ..., Z_k$, all with $\dim(Z_i)<\dim(X)$. Furthermore, for each $i$ the restriction of $F$ to $U\cap Z_i$ has finite fibres (vacuously if $U\cap Z_i=\emptyset$), and so by the induction hypothesis we have a bound of the form
\[h_{X, L}(u)=h_{Z_i, L|_{Z_i}}(u)\leq C_{1, i}h_{Y, M}(F(u))+C_{2, i}\]
for all $u\in U\cap Z_i$. It follows that
\[h_{X, L}(u)\leq \left(\max_{0\leq i\leq k}C_{1, i}\right)h_{Y, M}(F(u))+\left(\max_{0\leq i\leq k}C_{2, i}\right)\]
for all $u\in U$.
\end{proof}

\begin{proof}[Proof of Lemma~\ref{lem:sigmapullback}]
This is now a direct application of the previous lemma. We have chosen a projective $X\supseteq \M_d$ and an ample line bundle $L$ relative to which we are defining $h_{\M_d}$. Let $U\subseteq \M_d$ be the complement of the Latt\`{e}s locus. For some $n$, the morphism $\sigma_n:U\to S^{d^n+1}\PP^1$ has finite fibres (see~\cite[Corollary~2.3]{mcmullen}, and the comments following the proof).
\end{proof}

We are now in a position to prove the main result.
\begin{proof}[Proof Theorem~\ref{th:main}]
As mentioned in the introduction,  the inequality $\hcrit\ll  h_{\M_d}$ is found in~\cite{barbados}. We are concerned only with the other direction.

Suppose that $f\in \Hom_d$ has a fixed point with multiplier $\lambda$. By Lemma~\ref{lem:firstbound} there are constants $c_1$ and $c_2$ depending just on $d$ with
\begin{equation}\label{eq:lambdabound}d^{k+1}\hcrit(f)\geq (k-1)h(\lambda)-c_1h_{\Hom_d}(f)-c_2k\end{equation}
for any $k\geq 1$.
Summing over all fixed points, we have
\[(d+1)d^{k+1}\hcrit(f)\geq (k-1)h_{S^{d+1}\PP^1}\circ\sigma(f)-c_1(d+1)h_{\Hom_d}(f)-c_2(d+1)k.\]
Applying Lemma~\ref{lem:silv}, and assuming without loss of generality that $f$ is of minimal height in its conjugacy class (which we also denote $f$), we obtain
\begin{equation}\label{eq:sigmaonebound}(d+1)d^{k+1}\hcrit(f)\geq (k-1)h_{S^{d+1}\PP^1}\circ\sigma(f)-c_3h_{\M_d}(f)-c_4k,\end{equation}
for new constants $c_3$ and $c_4$ depending just on $d$, in particular since $h\circ\sigma(f)$ and $\hcrit(f)$ are constant on conjugacy classes.
Note that an inequality of this form holds for each $d$.

We claim that $h_{\M_{d^n}}(f^n)\ll h_{\M_d}(f)$, with constants depending on $d$ and $n$. This is just because $n$-fold iteration defines a morphism of affine varieties $\M_d\to\M_{d^n}$ (for example, see~\cite[Introduction]{jhsequi}), but we can also see this more directly from Lemma~\ref{lem:silv}. Specifically, For any $f$ we have by \cite[Proposition~5d]{hs} that
\[h_{\Hom_{d^n}}(f^n)\leq \left(\frac{d^n-1}{d-1}\right)h_{\Hom_d}(f)+O(d^n),\]
which is obtained by estimating the coefficients of $f^n$ above in terms of those of $f$. Lemma~\ref{lem:silv} tells us that,  without loss of generality, we may replace $f$ with a conjugate so that $h_{\Hom_d}(f)\ll h_{\M_d}(f)$, and then note that
\[h_{\M_{d^n}}(f^n)\ll h_{\Hom_{d^n}}(f^n)\ll h_{\Hom_d}(f)\ll h_{\M_d}(f),\]
with implied constants depending on $d$ and $n$, and with the first inequality following from an application of Lemma~\ref{lem:silv} to $\M_{d^n}$ (which does not require replacing $f^n$ by a conjugate).

Also, note that the chain rule and properties of the canonical height give $\hcrit(f^n)=n\hcrit(f)$. It follows that, applying~\eqref{eq:sigmaonebound} to $f^n$, we have (for any $n$)
\[(d^n+1)d^{n(k+1)}n\hcrit(f)\geq (k-1)h_{S^{d^n+1}\PP^1}\circ\sigma_n(f)-c_5h_{\M_d}(f)-c_6k,\]
where $c_5$ and $c_6$ are constants now depending on both $d$ and $n$. By Lemma~\ref{lem:sigmapullback}, there exists an $n$, which we now fix, and constants $\epsilon>0$ and $c_7$ such that \[h_{S^{d^n+1}\PP^1}\circ\sigma_n(f)\geq \epsilon h_{\M_d}(f)-c_7\]
for all non-Latt\`{e}s $f\in \M_d$.
From this we have, for any $k\geq 1$,
\[\hcrit(f)\geq \left(\frac{(k-1)\epsilon - c_5}{(d^n+1)d^{n(k+1)}n}\right)h_{\M_d}(f)-c_8,\]
for some $c_8$ depending just on $k$ (and our fixed $d$ and $n$). Choosing $k>1+c_5/\epsilon$ gives the inequality in the statement of Theorem~\ref{th:main}.
\end{proof}

The proof of Theorem~\ref{th:fibration} is now quite quick.
\begin{proof}[Proof of Theorem~\ref{th:fibration}]
Let $\operatorname{Per}_n(\lambda)\subseteq \M_d$ be the subvariety consisting of conjugacy classes of rational functions admitting an $n$-cycle of multiplier $\lambda$. Note that, by the main result of~\cite{bijl}, if $\hcrit$ vanishes at all on $\operatorname{Per}_n(\lambda)$, then $h(\lambda)$ is bounded in terms of $d$ and $n$. Taking $B$ at least as large as this bound and $h(\lambda)> B$, we will assume that $\operatorname{Per}_n(\lambda)$ contains no PCF maps, and in particular no Latt\`{e}s examples.

By~\eqref{eq:lambdabound} combined with Lemma~\ref{lem:silv} and the estimate $h_{\M_{d^n}}(f^n)\ll h_{\M_d}(f)$ from the proof of Theorem~\ref{th:main}, we have
\[(k-1)h(\lambda)\leq d^{nk+1}n\hcrit(f)+c_1h_{\M_d}(f)+c_2k\]
for some constants depending on $d$ and $n$. On the other hand, Theorem~\ref{th:main} now gives $h_{\M_d}(f)\ll \hcrit(f)$ (recall that we have ensured that $\operatorname{Per}_n(\lambda)$ contains no Latt\`{e}s examples), and so taking $k=2$ we have
\[h(\lambda)\leq c_3\hcrit(f)+c_4,\]
with $c_3$ and $c_4$ dependent on $d$ and $n$. This also gives
\[c_{3}^{-1}\leq \frac{\hcrit(f)}{h(\lambda)}+o(1)\]
with $o(1)\to 0$ 
as $h(\lambda)\to\infty$, proving the Theorem.
\end{proof}

For the claim made immediately after the statement of Theorem~\ref{th:fibration}, we consider the constants somewhat more carefully. In particular, if $P$ is a point of period $n$ for $f$, with multiplier $\lambda$, then it is a fixed point of $f^{nm}$ with multiplier $\lambda^m$. By~\eqref{eq:epsilonsum}, we see that for $d\geq 3$ we may take
\[c_2(d)=(d-1)\log 3+d\log\operatorname{lcm}(1, ..., d)\leq 1.04d^2+(d-1)\log 3,\]
by an estimate of 
Rosser and Schoenfeld~\cite{rs} (noting that $\log\operatorname{lcm}(1, ... ,d)$ is the second Tchebyshev function from the proof of the Prime Number Theorem). Again by~\eqref{eq:epsilonsum} we may take $c_2(2)=5\log 2\leq 1.04(2)^2+(2-1)\log 3$ as well. Thus, applying the estimates above to $f^{nm}$, we have
\[\left((k-1)h(\lambda^m)-kc_2(d^{nm})\right)\leq (d^{k+1}nm+c_3)\hcrit(f)+c_4,\]
where $c_3$ and $c_4$ depend on $n$, $m$, and $d$. Now fix $m$, let $\delta>0$, and suppose that $\hcrit$ admits no positive lower bound on $\operatorname{Per}_n(\lambda)$. Then we must in fact have
\[\left((k-1)h(\lambda^m)-kc_2(d^{nm})\right)\leq c_4,\]
for each $k$, which is possible only if
\[h(\lambda)\leq \frac{k}{m(k-1)}c_2(d^{nm})\]
for all $k$, or in other words $h(\lambda)\leq \frac{1}{m}c_2(d^{nm})$.


\section{Super-attracting fixed points}\label{sec:sa}

The arguments in Section~\ref{sec:global} make use of McMullen's Theorem on the multiplier spectrum, as well as estimates relating $h_{\Hom_d}$ to $h_{\M_d}$, both of which interfere with the presentation of explicit constants. If we are willing to restrict attention to rational functions with a super-attracting fixed point, we may avoid any inexplicit estimates. In Subsection~\ref{subsec:local} we present local estimates that play the role of those in Section~\ref{sec:greens} but that, in this context, make no reference to multipliers. In Subsection~\ref{subsec:global} we sum these over all places to obtain Theorem~\ref{th:geom}.

\subsection{Local estimates}\label{subsec:local}
As in Section~\ref{sec:greens}, we will assume that $K$ is an algebraically closed field of characteristic $0$ or $p>d$, complete with respect to some absolute value $|\cdot|$ corresponding to the valuation $v$. In Subsection~\ref{subsec:global} below, quantities depending on $v$ will acquire a subscript.

We restrict attention to  $f$ of the form
\begin{equation}\label{eq:sanormform}f(z)=\frac{z^e+\cdots +a_dz^d}{1+\cdots +b_{d-1}z^{d-1}}=\frac{z^e\prod_{i=1}^{d-e}(1-\alpha_iz)}{\prod_{j=1}^{d-1}(1-\beta_jz)},\end{equation}
with $2\leq e\leq d$.
Note that, unlike in the previous section, the normal form is not maintained under a scaling $f(z)\mapsto \xi^{-1}f(\xi z)$, and so we will need to keep more careful track of the $\alpha_i$ and $\beta_j$.
For such $f$, define \[\rho_f=\frac{1}{\max\{|\alpha_1|, ..., |\alpha_{d-e}|, |\beta_1|,... , |\beta_{d-1}|\}},\] and a constant $C_v$ by
\[C_v=\begin{cases}
\left(\frac{2e+1}{e-1}\right)\log 2 + \frac{d-e}{e-1}\log 3& \text{ if $v$ is archimedean}\\
\frac{d}{e-1}\log\max_{1\leq m\leq d}|m^{-1}| & \text{ if $v$ is non-archimedean}.
\end{cases}
\]

\begin{lemma}\label{lem:sabranch}
Suppose that $\log\rho_f+C_v<0$. Then there is a branch point $\beta$ of $f$ satisfying
\[-\infty<\log|f^k(\beta)|<e^k\log \rho_f+\frac{e^k}{e-1}\log^+|2^{e-1}3^{d-e}|\]
for all $k\geq 1$.
\end{lemma}

\begin{proof}[Proof of Lemma~\ref{lem:sabranch} over $\CC$]
Note that for any $0\neq |z| < \frac{1}{2}\rho_f$ and $|z|<\exp(-C_v)$,
\[0\neq |f(z)|=|z|^e\frac{\prod_{i=1}^{d-e}|1-\alpha_i z|}{\prod_{j=1}^{d-1} |1-\beta_j z|}\leq|z|^e\frac{(3/2)^{d-e}}{1/2^{d-1}}\leq |z|^e3^{d-e}2^{e-1}\leq |z|,\]
and so by induction
\begin{equation}\label{eq:branchboundsa}\log|f^k(z)|\leq e^k\log|z|+\frac{e^k-1}{e-1}\log|3^{d-e}2^{e-1}|<e^k\left(\log|z|+\frac{1}{e-1}\log|3^{d-e}2^{e-1}|\right).\end{equation}
It suffices to show that $D(0, \frac{1}{2}\rho_f)$ contains a branch point of $f$ other than $z=0$.

Suppose that $f$ has no branch points in $D(0, \frac{1}{2}\rho_f)$, other than at $z=0$, and let $W$ be the connected component of $f^{-1}(D(0, \frac{1}{2}\rho_f))$ containing $0$. Topologically, $W$ is a disk with some number of punctures, but since $f:W\setminus f^{-1}(0)\to D(0, \frac{1}{2}\rho_f)\setminus\{0\}$ is unbranched and $f(0)=0$, we see that $W$ is simply connected, and $f^{-1}(0)\cap W =\{0\}$. Since $f(W)\subseteq D(0, \frac{1}{2}\rho_f)$, $W$ contains no poles of $f$. In particular, $W$ does not contain $\overline{D(0, \rho_f)}$, and so by Koebe's $\frac{1}{4}$ Theorem, the conformal radius of $W$ relative to 0 is no greater than $4\rho_f$.

On the other hand, $f:W\to D(0, \frac{1}{2}\rho_f)$ factors through an analytic $e$th root $\beta:W\to D(0, (\frac{1}{2}\rho_f)^{1/e})$ given by $\beta(z)=z+O(z^2)$. This map witnesses  $W$ having conformal radius exactly $(\frac{1}{2}\rho_f)^{1/e}$, and so
\[\left(\frac{1}{2}\rho_f\right)^{1/e}\leq 4\rho_f,\]
or $\rho_f\geq 2^{-(2e+1)/(e-1)}$.
This contradicts our hypothesis that $\log\rho_f<-C_v$.
\end{proof}

\begin{proof}[Proof of Lemma~\ref{lem:sabranch} over non-archimedean fields]
The proof closely follows \cite{bijl}. 

 Let $U\subseteq \PP^1_{\mathrm{Berk}}$ be the open disk at $0$ of radius $\rho_f$, and let $V\supseteq U$ be the connected component of $f^{-1}(U)$ containing $0$. We know that $f:V\to U$ is $m$-to-1, for some $m\geq e=e_{f}(0)$ which we fix now. We also know that $V$ is an open affinoid, that is, $V=D(0, R)\setminus (W_1\cup\cdots \cup W_k)$ for some closed disks $W_i=\overline{D(b_i, R_i)}\subseteq D(0, R)$ with $W_i\cap U=\emptyset$.

For $\zeta\in \PP^1_{\mathrm{Berk}}$ we set (as in~\cite{bijl})
 \[\operatorname{rad}(\zeta)=\inf_{a\in \PP^1(K)}\|z-a\|_\zeta\]
the distortion
\[\delta(f, \zeta)=\log\operatorname{rad}(\zeta)+\log\|f'\|_\zeta-\log\|f\|_\zeta\]
and
\[G(\zeta)=m\delta(f, \zeta)+\log\|f\|_\zeta.\]
Also, for any point $\zeta_{a, t}$ corresponding to a disk, let
\begin{gather*}
N^+(f, \zeta_{a, t}, b)=\#\{z\in \overline{D(a, t)}:f(z)=b\}\\
N^-(f, \zeta_{a, t}, b)=\#\{z\in D(a, t):f(z)=b\}.\\
\end{gather*}
Note that $t\mapsto G(\zeta_{0, t})$ is continuous and piecewise linear in $\log t$, with slope 
\begin{equation}\label{eq:slope}m(1+N^{\pm}(f', \zeta_{0, t}, 0)-N^{\pm}(f', \zeta_{0, t}, \infty))+(1-m)(N^{\pm}(f, \zeta_{0, t}, 0)-N^{\pm}(f, \zeta_{0, t}, \infty))\end{equation}
except at the points where the slope is undefined (see~\cite[Proof of Theorem~4.1]{bijl}; note that the points at which the slope is undefined are exactly those at which there is a distinction between $N^+$ and $N^-$).

Note that 
\begin{equation}\label{eq:deltabound}\log|N^\pm (f, \zeta, 0)-N^\pm (f, \zeta, \infty)|\leq \delta(f, \zeta)\leq 0,\end{equation}
by~\cite[Lemma~3.3]{bijl},
and so we have 
\[G(\zeta_{0, \rho_f})\leq \log\|f\|_{\zeta_{0, \rho_f}}=e\log \rho_f.\]
On the other hand, we have  $f(\zeta_{0, R})=f(\zeta_{b_i, R_i})=\zeta_{0, \rho_f}$, and so
\[\log\|f\|_{\zeta_{0, R}}=\log\|f\|_{\zeta_{b_i, R_i}}=\log \rho_f.\]
Again by~\eqref{eq:deltabound}
\[G(\zeta_{0, R})\geq m\min_{1\leq k\leq d}\log|k|+\log \rho_f\geq \log \rho_f-\frac{m(e-1)}{d} C_v,\]
and hence
\[G(\zeta_{0, R}) -G(\zeta_{0, \rho_f})\geq (e-1)\log \rho^{-1}_{f}-\frac{m(e-1)}{d}C_v>0\]
by our hypothesis on $\rho_f$ (and since $m\leq d$). Since the function $\log t\mapsto G(\zeta_{0, t})$ increases on average from $t=\rho_f$ to $t=R$,
 there exist $t\in [\rho_f, R)$ where the graph has positive slope (given by ~\eqref{eq:slope}). We take $S$ to be the infimum of such $t$.  We have
\begin{equation}\label{eq:nplus}
m(1+N^+(f', \zeta_{0, S}, 0)-N^+(f', \zeta_{0, S}, \infty))+(1-m)(N^+(f, \zeta_{0, S}, 0)-N^+(f, \zeta_{0, S}, \infty))\geq 1,
\end{equation}
since $t\mapsto N^+(g, \zeta_{0, t}, b)$ is upper semi-continuous for any $g$ and $b$, and
therefore the quantity on the left is both positive and an integer. Note that $G(\zeta_{0, S})\leq G(\zeta_{0, \rho_f})$, or else the same argument again gives a $t<S$ at which the graph of $t\mapsto G(\zeta_{0, t})$ has positive slope, contradicting the construction of $S$.

We now discard any $W_i$ with $b_i\not\in \overline{D(0, S})$, renumbering so that now $W_1, ..., W_k$ remain. Note that for each $i$
\[G(\zeta_{b_i, R_i})-G(\zeta_{0, S})\geq (e-1)\log\rho_f^{-1}-\frac{m(e-1)}{d}C_v>0\]
just as above. We take $S_i$ to be the supremum of the (nonempty) set of $t\in (R_i, S]$ on which  $\log t\to G(\zeta_{b_i, t})$ is decreasing, so that
\begin{equation}\label{eq:nminus}m(1+N^-(f', \zeta_{b_i, S_i}, 0)-N^-(f', \zeta_{b_i, S_i}, \infty))+(1-m)(N^-(f, \zeta_{b_i, S_i}, 0)-N^-(f, \zeta_{b_i, S_i}, \infty))\leq -1\end{equation}
by the lower-semicontinuity of $N^-(g, \zeta_{b_i, t}, b)$ in $t$.

Now set
\[W=\overline{D(0, S)}\setminus (D(b_1, S_1)\cup\cdots\cup D(b_k, S_k))\] and
let $N(g, W, b)$ count solutions to $g(z)=b$ in $W$. 
Exactly as in~\cite{bijl}, we subtract from~\eqref{eq:nplus} the sum of~\eqref{eq:nminus} for $1\leq i\leq k$ to obtain
\begin{equation}\label{eq:kplusone}m((1-k)+N(f', W, 0)) +(1-m)N(f, W, 0)\geq 1+k,\end{equation}
since $f$ and $f'$ have no poles in $W$.
Exactly as in~\cite{bijl}, we now see that $W$ contains a critical point $\zeta$ with $f(\zeta)\neq 0$. In particular, if $W$ contains $A$ critical points that are not roots of $f$ and $B$ \emph{distinct} roots of $f$, then 
\begin{equation*}A+N(f, W, 0)=N(f', W, 0)+B.\end{equation*}
Isolating $N(f', W, 0)$  and using this value in~\eqref{eq:kplusone}, we obtain
\begin{equation}\label{eq:berkobound}m(1-k+A-B)+N(f, W, 0)\geq 1+k\end{equation}

If $k=0$ then $W=\overline{D(0, S)}$. Since $W$ contains no poles of $f$, but $\overline{D(0, S)}\supset D(0, \rho_f)$ contains either a nonzero root or a pole, it follows that in this case $B\geq 2$. Here~\eqref{eq:berkobound} becomes
\[m(1+A-B)+N(f, W, 0)\geq 1,\]
and $A$ must be positive, since $N(f, W, 0)\leq m$.

In general, using $N(f, W, 0)\leq m$ and $B\geq 1$, the inequality~\eqref{eq:berkobound} implies
\[m(1-k+A)\geq 1+k,\]
from which $mA\geq 2$ when $k\geq 1$. Again we have $A>0$.

We have shown that there is critical point $\zeta\in W$ with $f(\zeta)\neq 0$, and so there is a branch point $\beta=f(\zeta)\in U\setminus\{0\}$. But for any $z\in U\setminus \{0\}$ we have
\[\log|f^k(z)|=e^k\log|z|<e^k\log \rho_f\]
for all $k\geq 1$, by the ultrametric inequality.
\end{proof}

The following estimate plays the role of Lemma~\ref{lem:key} in the present case. As in Section~\ref{sec:greens}, take $B_f'$ to be that part of the branch locus whose forward orbit does not contain 0, that is,
\[B_f'=\sum_{P\not\in\mathcal{O}_f^-(0)}\left(e_P(f)-1\right)[f(P)].\]

\begin{lemma}\label{lem:saest}
For all $f$ in the form~\eqref{eq:sanormform},
\begin{multline}\label{eq:saest}
g_f(f_*^kB_f', 0)\\+\deg(B_f')\left(\frac{1}{d-1}\log^+|2(2d-1)!|+\frac{2d-1}{d-1}\log\|f\|-(d-1)r(f)\right)\\\geq \frac{e^k}{d-1}\log\|f\|-e^k\left(C_v+2\log^+|2|+\frac{d-e}{e-1}\log^+|3|\right),
\end{multline}
for all $k\geq 1$.
\end{lemma}

\begin{proof}
We first remark that we may apply the estimates in Lemma~\ref{lem:greensineq}, which nowhere used the hypothesis that $\lambda\neq 0$ in the normal form~\eqref{eq:normalform}.

We will begin by showing that
\begin{multline}\label{eq:interest}
g_f(f_*^kB_f', 0)\\+\deg(B_f')\left(\frac{1}{d-1}\log^+|2(2d-1)!|+\frac{2d-1}{d-1}\log\|f\|-(d-1)r(f)\right)\\\geq e^k(\log^+\rho_f^{-1}-C_v)-\frac{e^k}{e-1}\log^+|3^{d-e}2^{e-1}|.
\end{multline}
Note that by Lemma~\ref{lem:greensineq}, the left-hand side is non-negative, and so~\eqref{eq:interest} is immediate when $\log^+\rho_f^{-1} -C_v\leq 0$. To treat the other case, we will assume that
\begin{equation}\label{eq:rhobig}\log\rho_f^{-1}=\log^+\rho_f^{-1} > C_v\geq 0.\end{equation}
It follows from Lemma~\ref{lem:sabranch} that there is a branch point $\beta$ satisfying~\eqref{eq:branchboundsa} for all $k\geq 1$. By Lemma~\ref{lem:greensineq}, the left-hand-side of~\eqref{eq:interest} is bounded below by
\[\log^+\left|\frac{1}{f^k(\beta)}\right|\geq \log\left|\frac{1}{f^k(\beta)}\right|\geq e^k\log|\rho_f^{-1}|-\frac{e^k}{e-1}\log^+|3^{d-e}2^{e-1}|.\]
Now~\eqref{eq:interest} follows from the fact that $C_v\geq 0$.

We are left with deducing the estimate in the statement of the lemma from that in~\eqref{eq:interest}.
We may apply \eqref{eq:rootslower} from Lemma~\ref{lem:newtonP} to the numerator and denominator of~\eqref{eq:sanormform}  (or rather to their reciprocal polynomials, whose roots are the $\alpha_i$ and $\beta_j$)  to obtain
\[\log^+\|a_i\|\leq (d-e)\log^+\|\alpha_i\|+(d-e)\log^+|2|\]
and
\[\log^+\|b_j\|\leq (d-1)\log^+\|\beta_j\|+(d-1)\log^+|2|,\]
from which
\[\log\|f\|\leq (d-1)\log\rho_f^{-1}+(d-1)\log^+|2|\]
(recalling that $\log\rho_f^{-1}>0$ from \eqref{eq:rhobig}). Combining this with~\eqref{eq:interest} gives
the intended estimate.
\end{proof}

\subsection{Global estimates}\label{subsec:global}

For the proof of Theorem~\ref{th:geom}, we let $K$ be a number field, and $M_K$ its standard set of absolute values, normalized as in Section~\ref{sec:global}. For each place $v$ we will apply the estimates in Subsection~\ref{subsec:local} over $\CC_v$, the completion of the algebraic closure of the completion of $K$ with respect to $v$, and resulting quantities will acquire a subscript $v$.

\begin{proof}[Proof of Theorem~\ref{th:geom}]
We sum the estimate from Lemma~\ref{lem:saest} over all places. Note that
\[\sum_{v\in M_K}\frac{[K_v:\QQ_v]}{[K:\QQ]}C_v =  \left(\frac{2e+1}{e-1}\right)\log 2 + \frac{d-e}{e-1}\log 3+\frac{d}{e-1}\log\operatorname{lcm}(1, 2, ..., d),
\]
and recall the sums~\eqref{eq:hcritsum}, \eqref{eq:homdsum}, \eqref{eq:Nsum} from the proof of Lemma~\ref{lem:firstbound}.
From these, summing~\eqref{eq:saest} over all places gives
\begin{equation*}d^{k+1}\hcrit(f)\geq  \left(\frac{e^k}{d-1} - \frac{(2d-1)\deg(B'_f)}{d-1}\right)h_{\Hom_d}(f) - E(d, e, k)
\end{equation*}
for 
\begin{multline*} 
E(d, e, k)=\frac{\deg(B'_f)}{d-1}\log(2(2d-1)!)\\+e^k\left(\frac{4e-1}{e-1}\log 2+\frac{2(d-e)}{e-1}\log 3 +\frac{d}{e-1}\log\operatorname{lcm}(1, 2, ..., d)\right).
\end{multline*}
Note also that since $0$ is a critical point of multiplicity $e-1$, $\deg(B'_f)\leq 2d-e-1$.

We now fix $k$ so that \[(2d-1)(2d-e-1)+1\leq e^k< e((2d-1)(2d-e-1)+1),\]
from which we get
\begin{gather*}
\frac{e^k}{d-1} - \frac{(2d-1)\deg(B'_f)}{d-1}\geq \frac{1}{d-1}\\\intertext{and}
d^{k+1}< d^{2+\log_e((2d-1)(2d-e-1)+1)}= d^2((2d-1)(2d-e-1)+1)^{\log_e d}.
\end{gather*}
This value of $k$ gives us
\[\hcrit(f)\geq \frac{1}{(d-1)d^2(4d^2-2(e+2)d+e+2)^{\log d/\log e}}h_{\Hom_d}(f)-C_{d, e},\]
with
\begin{multline}\label{eq:exp}
C_{d, e}=\frac{2d-e-1}{d^2(d-1)(4d^2-2(e+2)d+e+2)^{\log d/\log e}}
\log(2(2d-1)!) \\+\frac{e}{d^2(d-1)(4d^2-2(e+2)d+e+2)^{(\log d/\log e)-1}}\Big(\frac{2(d-e)}{e-1}\log 3+\frac{4e-1}{e-1}\log 2\\ \frac{d}{e-1}\log\operatorname{lcm}(1, 2, ..., d)\Big).
\end{multline}
Note that $\log\operatorname{lcm}(1, ..., d)$ is the second Chebyshev function and by explicit estimates in the direction of the Prime Number Theorem by Rosser and Schoenfeld~\cite{rs}, we have $\log\operatorname{lcm}(1, ..., d)< 1.04d$.
\end{proof}

\begin{remark}\label{rm:ff}
In this section we have been working over a number field, but the estimates in Subsection~\ref{subsec:local} did not depend on the origin of the local field under consideration (although we did require characteristic 0 or $p>d$). If $k$ is an algebraically closed field, and $X/k$ is a normal projective variety, then there is a set of places $M$ on the function field $K=k(X)$ corresponding to irreducible divisors on $X$. If  \[|x|_v=e^{-\ord_v(x)\deg(v)},\] where $\ord_v$ is the order of vanishing along the divisor corresponding to $v$ and $\deg(v)$ is the degree of this divisor  (relative to some fixed ample class on $X$), we set
\[h(P)=\sum_{v\in M}\log\|P\|_v\]
for $P\in \PP^N(K)$.
The points of height $0$ are precisely the points with constant coordinates \cite[\S1.4 and Example~2.4.11]{bg}.

Note that in the case of a non-archimedean absolute value which is not $p$-adic for any $p\leq d$, many of the terms in~\eqref{eq:saest} vanish (for instance, $C_v=0$ and $\log^+|N|=0$ for all $N$). If all of our places satisfy this condition, as is the case for a function field, the sum over all places in the proof of Theorem~\ref{th:geom} simplifies significantly, to
\[\hcrit(f)\geq \frac{1}{(d-1)d^2(4d^2-2(e+2)d+e+2)^{\log d/\log e}}h_{\Hom_d}(f).\]
In particular, if $K$ is the function field of a normal projective variety over an algebraically closed field of characteristic 0 or $p>d$, then any PCF family in the form~\eqref{eq:sanormform} must have constant coefficients. Note Levy has already shown that such an example is conjugate over $\overline{K}$ to a rational function with constant coefficients~\cite[Theorem~1.10]{levysa}, so the only novelty here is that we know that this particular conjugate is already defined over the constant field.
\end{remark}


\section{Quadratic morphisms}\label{sec:quad}

In this section we we turn our attention to morphisms of the form
\begin{equation}\label{eq:m2normalform}f(z)=\frac{\lambda_0 z + z^2}{\lambda_\infty z+1},\end{equation}
with $\lambda_0\lambda_\infty\neq 1$,
and give explicit estimates. Note that, over an algebraically closed field, every quadratic morphism is conjugate to one of this form \cite[Lemma   4.59, p.~190]{ads}, or one of the form $f(z)=z^{-1}+a+z$, the latter family being simpler to handle. For convenience, estimates are relative to the fixed point $z=\infty$ rather than $z=0$, and we note that in the form~\eqref{eq:m2normalform}, the parameters $\lambda_0$ and $\lambda_\infty$ are exactly the multipliers at the fixed points $z=0$ and $z=\infty$.

\subsection{Local estimates}
Let $K$ be an algebraically closed field, complete with respect to some absolute value $|\cdot|$. We work with the lift \begin{equation}\label{eq:lift}F(x, y)=(F_1(x, y), F_2(x, y))=(\lambda_0 xy+x^2, \lambda_\infty xy+y^2),\end{equation}
noting that
\[\operatorname{Res}(F_1, F_2)=1-\lambda_0\lambda_\infty.\]
\begin{lemma}
For all $z$,
\[g_f(z, \infty)\geq \log^+|z|-2\log\|1, \lambda_0, \lambda_\infty\|+r(F)-\log^+|2|.\]
\end{lemma}

\begin{proof}
Since
\[\lambda_\infty^2 yF_1(x, y)+(-\lambda_\infty x+(1-\lambda_0\lambda_\infty)y)F_2(x, y)=(1-\lambda_0\lambda_\infty)y^3\]
and
\[\lambda_0^2 xF_2(x, y)+(-\lambda_0 y+(1-\lambda_0\lambda_\infty)x)F_1(x, y)=(1-\lambda_0\lambda_\infty)x^3,\]
we have that (as in the proof of Lemma~\ref{lem:greensineq})
\[\log\|P\|\leq \frac{1}{2}\log\|F(P)\|+\frac{1}{2}\log^+\max\{|1-\lambda_0\lambda_\infty|, |\lambda_0|^2, |\lambda_\infty|^2\}-\frac{1}{2}\log|1-\lambda_0\lambda_\infty|.\]
So
\[\log\|P\|\leq H_F(P)+2\log\|1, \lambda_0, \lambda_\infty\|+\log ^+|2|-2r(F).\]
Since $F(1, 0)=(1, 0)$, $H_F(1, 0)=0$ and
\[g_f(z, \infty)\geq \log^+|z|-2\log\|1, \lambda_0, \lambda_\infty\|+r(F)-\log^+|2|.\]
\end{proof}

\begin{lemma}
Let
\[\epsilon=\begin{cases}
\sqrt{2}-1=0.4142... &\text{if $v$ is archimedean}\\
1/4 & \text{if $v$ is 2-adic}\\
1 & \text{otherwise.}
\end{cases}
\] 
Then for $f$ as in \eqref{eq:m2normalform}, with $0<|\lambda_\infty|<\epsilon$, there is a branch point $\beta$ with
\[|f^k(\beta)|\geq (|\lambda^{-1}_\infty|\epsilon)^k\]
for all $k\geq 1$.
\end{lemma}

\begin{proof}
We start with the case of $v$ archimedean. Note that for any $z$ with $|z|\geq (\sqrt{2}+1)\max\{|\lambda_0|, |\lambda_\infty^{-1}|\}$ we have
\[|f(z)|=|z|\cdot\left|\frac{\frac{\lambda_0}{z}+1}{\frac{1}{z}+\lambda_\infty}\right|\geq |z|\cdot|\lambda_\infty^{-1}|\cdot\frac{\sqrt{2}}{2+\sqrt{2}}>|z|,\]
and so $|f^k(z)|\geq |z|(|\lambda_\infty^{-1}|\epsilon)^k\geq 6(|\lambda_\infty^{-1}|\epsilon)^k$ for all $k$ by induction. It remains to show that some branch point satisfies this hypothesis. Although the branch points of $f$ are not rational functions in $\lambda_0$ and $\lambda_\infty$, symmetric polynomials in the branch points must be, and it is useful to have explicit expressions for these.

Let
\[W(x, y)=\lambda_\infty x^2+2xy+\lambda_0y^2,\]
so that $W$ is a homogenous form vanishing at the critical points of $f$ in $\PP^1$ (by the quotient rule), and recall the forms $F_1$ and $F_2$ from \eqref{eq:lift}. Note that since the resultant of two homogeneous forms vanishes just in case the forms have a common zero, the resultant
\[B(X, Y)=\Res(YF_1(x, y)-XF_2(x, y), W(x, y))\]
(as forms in the variables $x, y$ over the ring $\ZZ[X, Y, \lambda_0, \lambda_\infty]$) vanishes precisely at (some homogeneous coordinates for) the images under $f$ of the roots of $W$, i.e., at the branch points of $f$. We can compute this resultant explicitly as the determinant of a $4\times 4$ matrix, and see that
\[B(X, Y)=\lambda_\infty^2X+2(2-\lambda_0\lambda_\infty)XY+\lambda_0^2Y^2.\]
Assuming that $\lambda_\infty\neq 0$ and then dehomogenizing, we see that the branch points $\beta_1$ and $\beta_2$ of $f$ satisfy
\begin{gather}
\beta_1+\beta_2=\frac{2(\lambda_0\lambda_\infty-2)}{\lambda_\infty^2}\label{eq:sumofbeta}\\
\beta_1\beta_2=\left(\frac{\lambda_0}{\lambda_\infty}\right)^2.\label{eq:prodofbeta}
\end{gather}

Suppose,  contrary to what we are trying to prove, that we have $|\beta_1|, |\beta_2|\leq (\sqrt{2}+1)\max\{|\lambda_0|, |\lambda_\infty^{-1}|\}$, and for now suppose that $\lambda_0\neq 0$. 
Then~\eqref{eq:prodofbeta} gives
\[\left|\frac{\lambda_0}{\lambda_\infty}\right|^2=|\beta_1\beta_2|\leq (\sqrt{2}+1)^2\max\{|\lambda_0|, |\lambda_\infty^{-1}|\}^2,\]
and hence $1\leq (\sqrt{2}+1)\max\{|\lambda_\infty|, |\lambda_0^{-1}|\}$. Note that $|\lambda_\infty|<\sqrt{2}-1$ now implies $|\lambda_0^{-1}|>|\lambda_\infty|$. It then follows from~\eqref{eq:sumofbeta} that 
\[\left|\frac{2(\lambda_0\lambda_\infty-2)}{\lambda_\infty^2}\right|=|\beta_1+\beta_2|\leq 4\max\{|\lambda_0|, |\lambda_\infty^{-1}|\}= 4|\lambda_\infty^{-1}|\]
whereupon (recalling again that $|\lambda_\infty|<\sqrt{2}-1$ and $|\lambda_0\lambda_\infty|<1$)
\[1\leq 2-|\lambda_0\lambda_\infty|\leq |\lambda_0\lambda_\infty-2|\leq 2|\lambda_\infty|<2(\sqrt{2}-1)=0.8284...<1,\]
a contradiction. On the other hand, if $\lambda_0=0$ then a contradiction follows directly from~\eqref{eq:sumofbeta}.

For a non-archimedean $v$, note that
\[|f(z)|=|\lambda_\infty^{-1}z|\]
as soon as $|z|>\max\{|\lambda_0|, |\lambda_\infty^{-1}|\}$.
Suppose both branch points fail this, so that
\[\left|\frac{\lambda_0}{\lambda_\infty}\right|^2=|\beta_1\beta_2|\leq\max\{|\lambda_0|, |\lambda_{\infty}^{-1}|\}^2,\]
and hence
\[1\leq\max\{|\lambda_\infty|, |\lambda_0^{-1}|\}^2.\]
Since $|\lambda_{\infty}|<1$, we must have $|\lambda_0|\leq 1$,
 from which we conclude that
\[|\lambda_0\lambda_\infty|\leq |\lambda_\infty|<\epsilon\leq |2|,\]
and therefore $|\lambda_0\lambda_\infty-2|=|2|$.
It follows that
\[\frac{|4|}{|\lambda_\infty^2|}= \left|\frac{2(\lambda_0\lambda_\infty-2)}{\lambda_\infty^2}\right|=|\beta_1+\beta_2|\leq |\lambda_\infty^{-1}|,\]
and hence $|\lambda_\infty|\geq |4|$. This contradicts our choice of $\epsilon$.
\end{proof}

As before, let $B_f'$ be that part of the branch locus whose support does not contain iterated preimages of $\infty$.
\begin{lemma}\label{lem:quadest}
For every $k\geq 1$,
\begin{multline*}g_f(f_*^kB_f', \infty)\geq k\log^+|\lambda_\infty^{-1}|+k\log \epsilon\\-\deg(B_f')\left(2\log\|1, \lambda_0, \lambda_\infty\|+\log^+|2|-r(F)\right).\end{multline*}
\end{lemma}

\begin{proof}
Combine the previous two lemmas.
\end{proof}

\subsection{Global estimates}
Now let $K$ be a number field, with its standard set of absolute values $M_K$. Quantities from the previous subsection gain a subscript.

\begin{proof}[Proof of Theorem~\ref{th:quad}]
Note that
\begin{gather*}
\sum_{v\in M_K}\frac{[K_v:\QQ_v]}{[K:\QQ]}\epsilon_v = -2\log 2 -\log (\sqrt{2}+1)\\
\intertext{and}
\sum_{v\in M_K}\frac{[K_v:\QQ_v]}{[K:\QQ]}\log\|1, \lambda_0, \lambda_\infty\|_v = h_{\PP^2}(\lambda_0, \lambda_\infty).
\end{gather*}
From the proof of Lemma~\ref{lem:firstbound} note also~\eqref{eq:hcritsum} with $d=2$, \eqref{eq:lambdasum} with $\lambda=\lambda_\infty$, \eqref{eq:Nsum} with $N=2$, and~\eqref{eq:ressum}. Combining these, summing the estimate from Lemma~\ref{lem:quadest} over all places, and using $\deg(B_f')\leq 2$, we have
 for any $k\geq 1$ that
\begin{equation}\label{eq:quad}
2^{k+1}\hcrit(f)\geq kh(\lambda_\infty) - k(2\log 2+\log (\sqrt{2}+1))-4h_{\PP^2}(\lambda_0, \lambda_\infty)-2\log 2.
\end{equation}
Note that conjugating $f$ by $1/z$ preserves the normal form~\eqref{eq:m2normalform}, but swaps $\lambda_0$ and $\lambda_\infty$. Since $\hcrit$ is conjugacy-invariant, we also have \eqref{eq:quad} with $\lambda_0$ and $\lambda_\infty$ swapped. Adding that inequality to \eqref{eq:quad} and  using the fact that
\[h_{\PP^2}(\lambda_0, \lambda_\infty)\leq h(\lambda_0)+h(\lambda_\infty),\]
we then have
\begin{equation}\label{eq:quadbound}2^{k+2}\hcrit(f)\geq (k-8)h_{\PP^2}(\lambda_0, \lambda_\infty)-4\log 2 - 2k(2\log 2+\log (\sqrt{2}+1)).\end{equation}
The estimate in Theorem~\ref{th:quad} is obtained by taking $k=10$ in~\eqref{eq:quadbound}.
\end{proof}

\begin{proof}[Proof of Corollary~\ref{cor:quad}]
Using~\eqref{eq:quadbound} to bound $h_{\PP^2}(\lambda_{0}, \lambda_\infty)$ from above in~\eqref{eq:quad} we have, for any positive integer $k\neq 8$,
\begin{equation}\label{eq:kbound}k(h(\lambda_\infty)-\log 12)\leq 2^{k+1}\left(1+\frac{2}{k-8}\right)\hcrit(f)+\frac{6k-12}{k-8}\log 2+\frac{2k}{k-8}\log (\sqrt{2}+1).\end{equation}
Now, if there exist, for a fixed $\lambda\neq 0$, quadratic morphisms $f$ of arbitrarily small critical height on $\operatorname{Per}_1(\lambda)$, we may write these in the normal form~\eqref{eq:m2normalform} with the given fixed point at $z=\infty$ (and hence $\lambda_\infty=\lambda$). Since $\hcrit(f)$ gets arbitrarily small for these examples, we conclude from~\eqref{eq:kbound} that
\[k(h(\lambda)-\log 12)\leq \frac{6k-12}{k-8}\log 2+\frac{2k}{k-8}\log (\sqrt{2}+1)<45\]
for each $k\geq 9$. Taking $k\to\infty$, we see a contradiction unless $h(\lambda)\leq \log 12$.
\end{proof}

\begin{remark}
Note that we neglected to handle the case $f(z)=z+a+z^{-1}$, but in fact it is not hard to prove the restriction of Theorem~\ref{th:main} to any one-parameter family. In general, if $U$ is a curve, and $f\in \Hom_d(U)$, we have by \cite[Theorem~4.1]{callsilv} that
\[\hat{h}_{f_t}(P_t)=(\hat{h}_f(P)+o(1))h_U(t)+O(1)\]
for any $P\in \PP^1(U)$, where $h_U$ is an degree-one Weil height on $U$, where $o(1)\to 0$ as $h_U(t)\to\infty$,  and where $\hat{h}_f(P)$ is the canonical height computed on the generic fibre. If all of the critical points are in $\PP^1(U)$, then we have
\begin{equation}\label{eq:family}\hcrit(f_t)=(\hcrit(f)+o(1))h_U(t)+O(1)\end{equation}
by summing over them. It then follows from a theorem of Baker~\cite{baker2} that $\hcrit(f_t)\asymp h_U(t)$ as long as the generic fibre $f$ is not isotrivial and not PCF. By Thurston Rigidity, the the family can only be PCF if it is  Latt\`{e}s or isotrivial. 

But even  if the critical points are not $U$-rational, they are $V$-rational for some finite cover $\varphi:V\to U$, and by functoriality of heights $\deg(\varphi)h_V=h_U\circ\varphi+O(1)$, giving again~\eqref{eq:family}.

Finally, if $f\in \Hom_d(U)$ is a non-constant, non-Latt\`{e}s family, then composing with the map $\Hom_d\to\M_d$ gives a map $[f]:U\to \M_d$ with image $\Gamma$ and finite fibres. We have, if $L$ is the ample class relative to which we are computing heights,
\[h_{\M_d}([f]_t)=h_{\Gamma, L\mid_\Gamma}([f]_t)= h_{U, F^*L\mid_\Gamma}(t)+O(1)=Ch_U(t)+O(1).\]
\end{remark}

\end{document}